\documentclass[12pt]{amsart}
\usepackage[margin=0.9in]{geometry}
\newcommand{\R}{\mathbb{R}}
\newcommand{\C}{\mathbb{C}}

\newcommand{\N}{\mathbb{N}}

\usepackage{calrsfs}

\renewcommand{\epsilon}{\varepsilon}

\usepackage[english]{babel}
\usepackage[alphabetic]{amsrefs}
\usepackage{amsmath,amssymb,amsfonts,amsthm,mathtools,enumerate,hyperref}
\usepackage{url}
\usepackage{graphicx,epstopdf,color}

\numberwithin{equation}{section}

\newtheorem{theorem}{Theorem}[section]
\newtheorem{lemma}[theorem]{Lemma}

\renewcommand{\le}{\leqslant}

\renewcommand{\ge}{\geqslant}

\begin{document}

\title[Vegetation--Rainfall--Bushfire Interactions]{Stability and
Self-Organized Patterns \\ in Coupled Ecohydrological--Fire Dynamics: \\ A Model of Vegetation--Rainfall--Bushfire Interactions}\thanks{This work
has been supported by the Australian Research Council Discovery Project {\em Non-local PDE approach to moving fronts and bushfires} DP250101080.}

\author[Serena Dipierro and Enrico Valdinoci]{
Serena Dipierro
and
Enrico Valdinoci}


\maketitle

{\scriptsize \begin{center} Department of Mathematics and Statistics\\
University of Western Australia\\ 35 Stirling Highway, WA6009 Crawley (Australia)
\end{center} }
\bigskip

{\scriptsize\begin{center}
E-mail addresses:
{\tt serena.dipierro@uwa.edu.au},
{\tt enrico.valdinoci@uwa.edu.au}
\end{center}}\bigskip

\begin{abstract}
This paper investigates the conditions for the stability and emergence of patterns in
a new three-component reaction-diffusion system. The system describes the coexistence
and interaction of water reservoirs, vegetation, and bushfire activity in a given ecosystem.

We perform a detailed stability analysis to determine the parameter space where an unstable homogeneous equilibrium becomes stable with respect to spatially nonuniform perturbations.

We also use diffusion to generate traveling trains in the form of periodic orbits of the linearized system. These orbits are remnants of an unstable equilibrium in the absence of diffusion and arise from a nonsingular eigenvalue crossing of the imaginary axis, while a third eigenvalue remains real and negative, thereby ensuring linear stability for monocromatic waves.

These phenomena differ from ``classical'' Turing and Hopf bifurcations, as the model does not involve distinct ``activators'' and ``inhibitors'', and the effects observed are not the byproduct of diffusion with necessarily differing speeds. Also, differently from the classical Turing pattern, the role of diffusion in this context is to
stabilize, rather than destabilize, homogeneous equilibria.

We also consider the case of plant competition, showing a suitable form of Turing instability for
slow-frequency oscillations in a small rainfall regime.
\end{abstract}

\section{Introduction}

In this paper, we introduce a mathematical model to describe an ecosystem
of water reservoirs, flammable vegetation, and wildfire, and we
investigate
whether diffusion can
stabilize equilibria and give rise to the spontaneous formation of persistent patterns.

To describe this ecosystem,
we consider three scalar functions of interest, namely bushfire intensity~$f$,
amount of vegetation~$v$, and water availability~$w$. These functions
may vary in time~$t\in\R$ and space~$x\in\R^n$.

\begin{figure}[h]
    \centering
    \includegraphics[height=4cm]{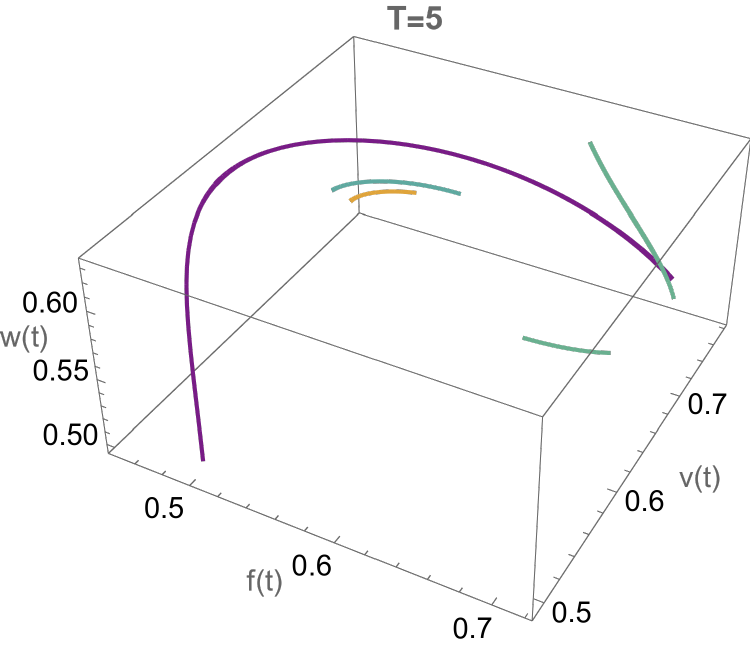}$\,$
    \includegraphics[height=4cm]{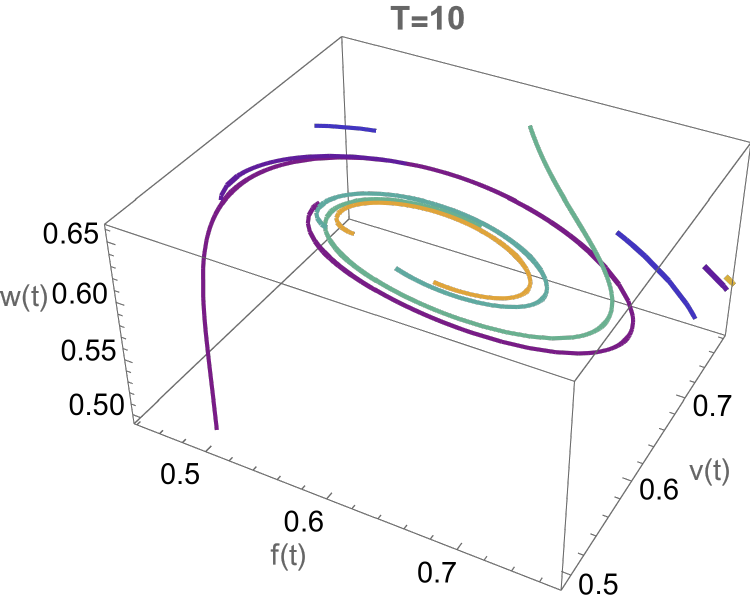}$\,$
    \includegraphics[height=4cm]{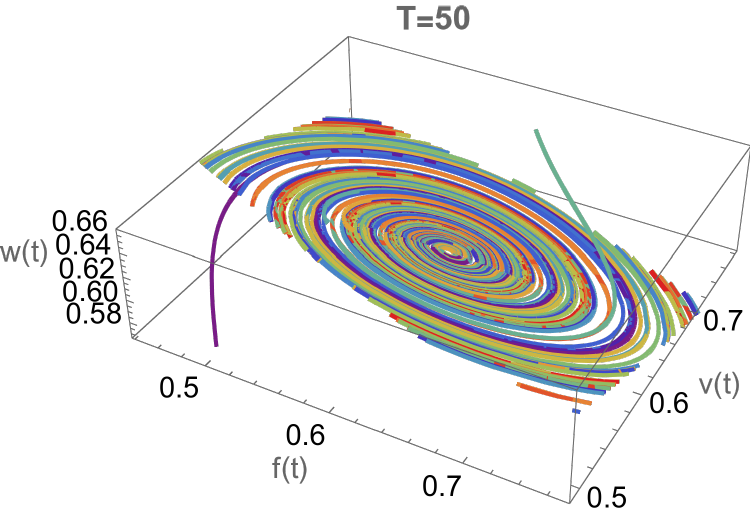}
    \caption{\footnotesize\sl Plot of 27 trajectories of~\eqref{ODE} corresponding to the parameters~$\alpha := \beta := \gamma := \delta := \epsilon := \eta := \zeta := 1$, with final time~$T\in\{5,10,50\}$.}
    \label{FIG100}
\end{figure}

We suppose that bushfires are fueled by the presence of vegetation and are
extinguished by water. Specifically, fire increases proportionally to
the current fire activity and to the availability of the vegetation (through
a proportionality factor~$\alpha\in(0,+\infty)$) 
and decreases proportionally to the current fire activity and the presence of water
in the terrain (through
a proportionality factor~$\beta\in(0,+\infty)$). We also assume that fire spread over the
land: for this, we model the fire diffusion via the Laplace operator (using a classical random
dispersal as a simple proxy for spreading; the diffusion coefficient will be denoted
by~$c\in(0,+\infty)$).

These assumptions lead to the equation
$$ \partial_t f=f(\alpha v-\beta w)+c\Delta f.$$
Notice that the product of~$fv$ (respectively, $fw$) on the right-hand side of the equation above
corresponds to a ``random encounter'' between fire and vegetation (respectively, fire and water).

Also, to model the evolution of vegetation, we suppose that it increases by the presence
of water in the terrain and decreased by fire activity (respectively, through proportionality
factors~$\zeta$ and~$\eta\in(0,+\infty)$). We also assume that the vegetation is ``static'',
namely it is not subject to diffusion. These ans\"atze can be translated into the equation
$$\partial_t v=v(\zeta w-\eta f).$$

Finally, water variations depend on rain (assumed to occur at a constant rate~$\gamma\in(0,+\infty)$), vegetation (sucking out water from the terrain with a proportionality factor~$\delta\in(0,+\infty)$), and evaporation (for this, we postulate that a proportion~$\epsilon$
of the available water evaporates in the unit of time). We also suppose that water diffuses
through the terrain (with diffusion coefficient equal to~$d\in(0,+\infty)$). These hypotheses lead to the equation
$$\partial_t w=\gamma-\delta v w-\epsilon w+d\Delta w.$$

\begin{figure}[h]
    \centering
    \includegraphics[height=4cm]{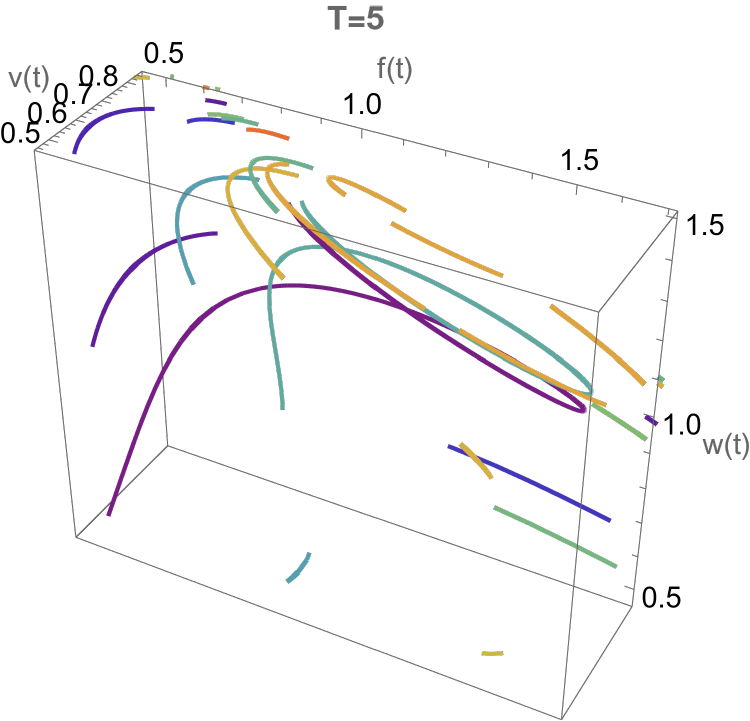}$\,$
    \includegraphics[height=4cm]{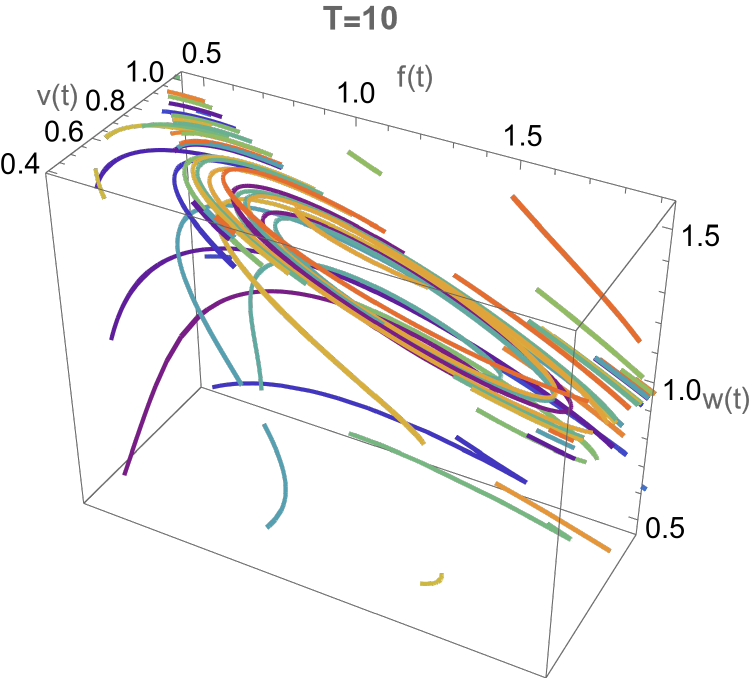}$\,$
    \includegraphics[height=4cm]{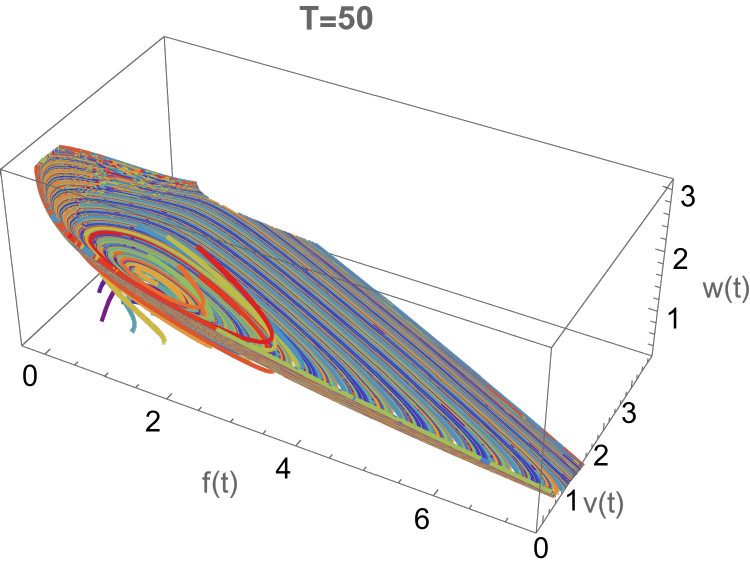}
    \caption{\footnotesize\sl Plot of 27 trajectories of~\eqref{ODE} corresponding to the parameters~$\alpha :=2$, $ \beta := \gamma := \delta :=  \eta := \zeta := 1$,  and~$\epsilon :=\frac1{10}$, with final time~$T\in\{5,10,50\}$.}
    \label{FIG101}
\end{figure}

We stress that, to keep the analysis as simple as possible, we have assumed the terrain to be flat and homogeneous (variations of fuel and slope can be taken into account, for example,
by replacing the Laplacian with an inhomogeneous diffusion operator).

The above equations can be summarized in the system
\begin{equation}\label{0dojc320yjyouj-3erfg}
\begin{dcases}
\partial_t f=f(\alpha v-\beta w)+c\Delta f, \\
\partial_t v=v(\zeta w-\eta f),\\
\partial_t w=\gamma-\delta v w-\epsilon w+d\Delta w.
\end{dcases}
\end{equation}

We mention that
it is also interesting to consider \label{ANTIP6} variations of the model in~\eqref{0dojc320yjyouj-3erfg} to account for the fact that plants compete for resources (such as water, nutrients) in a spatial neighborhood. This competition effect reduces vegetation growth where nearby biomass is high, providing a ``negative feedback'' on the variation of~$v$.
This additional competition feature will be incorporated
in the forthcoming equation~\eqref{NUOVO}: to initially maintain the analysis
as simple as possible, at this stage we focus on the simpler model
in~\eqref{0dojc320yjyouj-3erfg}.

An interesting simplification of the system of partial differential equations in~\eqref{0dojc320yjyouj-3erfg}
occurs in the absence of spatial diffusion, namely when~$c:=d:=0$. In this case, the problem can be considered as independent of the space variable~$x$ and boils down to the system of ordinary differential equations
\begin{equation}\label{ODE}
\begin{dcases}
\dot f=f(\alpha v-\beta w), \\
\dot v=v(\zeta w-\eta f),\\
\dot w=\gamma-\delta v w-\epsilon w.
\end{dcases}
\end{equation}

\begin{figure}[h]
    \centering
    \includegraphics[width=0.35\textwidth]{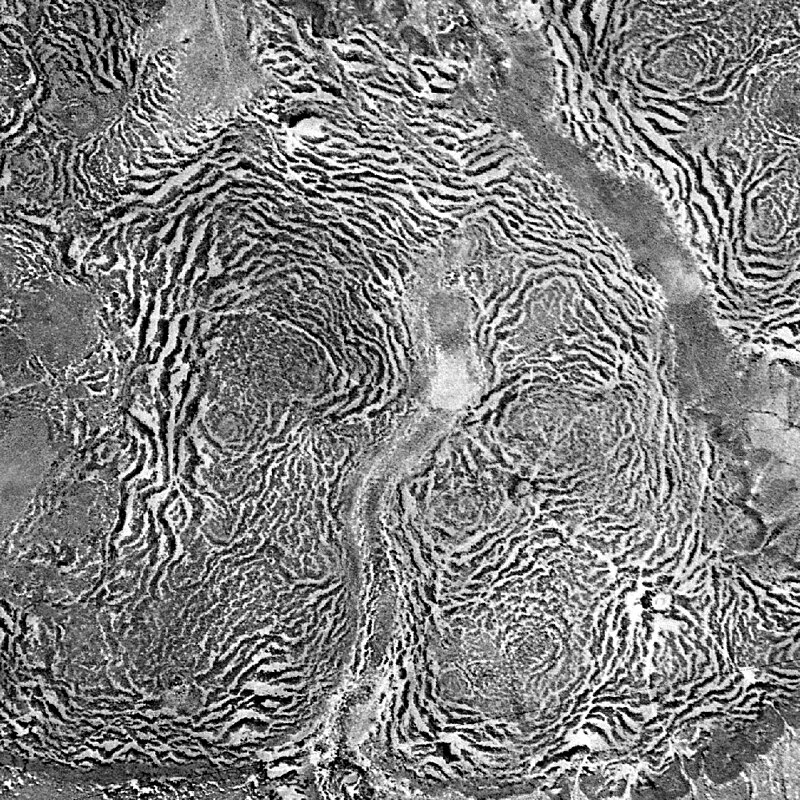}
    \caption{\footnotesize\sl A ``tiger bush'' plateau in Niger, with
    vegetation dominated by Combretum micranthum and Guiera senegalensis.
    United States Geological Survey (Public Domain).}
    \label{FIG1}
\end{figure}

We now motivate our analysis by presenting a series of concrete examples that illustrate the types of behaviors that
we seek to extract from the mathematical study of~\eqref{ODE}.
In this regard, we recall that, especially
in semi-arid ecosystems, vegetation often displays spatial self-organization, such as
stripes and spots, see Figure~\ref{FIG1}.
For further images about field observations of patterns in several types of vegetation, see~\cite[Figure 1 A--E]{HARD},
\cite[Figure~1 A--G]{RIET2},
and~\cite[Figure~1 A--H,
Figure~2 A--C,
Figure~3 A--C,
Figure~4 A--B]{BORGO}.
Examples of aerial images of water patterns are also displayed in Figure~\ref{FIGWA}.
For related patterns in different environments, see for example~\cite{WU05} for
landscapes in savanna parklands, \cite{HIEM} for krummholz patches and ribbon forests, and the references therein.

The mathematical study of self-organized vegetation patterns in semi-arid regions was also considered in~\cite{SHE}; the author deals with a system of two partial differential equations
for plant and water densities (presented in (1a) and (1b) there). This model was first introduced in~\cite{KLAUS}, 
is similar to but distinct from~\eqref{ODE} here, and does not account for fire activity.
Vegetation patches have been also analyzed in~\cite{HARD} via two partial differential equations
describing
biomass density
and ground water density (see in particular equations~(1) and~(2) there) to understand
desertification phenomena; in this setting, stability and bushfire activity
were not specifically considered, but the model predicted the coexistence
of intermediate
states between bare soil at low precipitation and homogeneous vegetation at high precipitation,
leading to formations in the form of spots, stripes, and holes.

Patterns are also known to appear in relation to bushfire activities, especially in terms of burn scars,
which are visible from satellite images relying on
visible and near-infrared light: the burn scars usually appear black, brown, or brick red and stand out in contrast to vegetation, which appears bright green, see Figure~\ref{FIG3}.

Mathematical models for plant biomass, soil water, and surface water were also considered in~\cite{RIET}.
These models were retaken in~\cite{HID}
in view of numerical simulations showing
long- range hyperuniformity of vegetation patterns (see Figures~A1--A4 there for aerial images of
vegetation patterns across the world);
in this setting, wildfires, together with droughts and overgrazing, are typically treated
as disturbances and not incorporated into the analysis.

See~\cite{SCAN, H78} for an analysis of the positive feedback between vegetation and resource accumulation, and an effort to identify the essential factors in the process in order to simplify the modeling accordingly.

See also~\cite{KEANE} for the design and
algorithms of simulations highlighting the dynamics of landscape fires and vegetation,
and~\cite{2015, 2017, 2022}
for the description of the dynamic processes that interact between wildland
fire and vegetation.
Data about the landscape mosaic of vegetation in comparison to rainfall and fire events were studied in~\cite{ETTEN},
especially in terms of time-series analysis.

Let us now dive into some technical aspect
of~\eqref{ODE}. First of all, we note that, while similar in some respects, this system of equations is structurally very different from the one for three competing species engaged in a rock-paper-scissors game that  has been used in mathematical biology, for example, to describe dynamics in lizards~\cite{SIN} and bacteria~\cite{KIR}. Indeed, a key distinction between the system of equations~\eqref{ODE} and the rock-paper-scissors game is the lack of symmetry. While rock-paper-scissors involves three interchangeable populations with symmetric, nontransitive rules, the roles of water availability~$w$, fire intensity~$f$, and vegetation~$v$ in~\eqref{ODE} are not equivalent
(not even in the absence of rain~$\gamma$ and evaporation~$\epsilon$): the asymmetry stems from the fact that fire does not\footnote{In fact, we mention that a more complex model could incorporate a decrease in water due to fire-enhanced evaporation (for example, by replacing $\epsilon$ with $\epsilon_1 + \epsilon_2 f$ in~\eqref{0dojc320yjyouj-3erfg} and~\eqref{ODE}), but we have chosen to keep our current model as simple as possible.}
increase the amount of water,
and correspondingly, the right-hand side of the last equation in~\eqref{ODE}
does not present any term proportional to~$fw$.

It is readily seen that the system of ordinary differential equations~\eqref{ODE} possesses the equilibria
\[
E_0 := \left(0, 0, \frac{\gamma}{\epsilon}\right)
\qquad{\mbox{and}}\qquad
E_1:=\left( f_\star,
v_\star , w_\star\right),\]
where\begin{equation}\label{EQSTvdARDAF23-4}
f_\star:=\frac{\zeta\,w_\star}{\eta},\qquad
v_\star:=\frac{\beta w_\star}{\alpha},\qquad{\mbox{and}}\qquad
w_\star := \frac{ \sqrt{\alpha^2\epsilon^2 + 4\alpha\beta\delta\gamma}-\alpha\epsilon}{2\beta\delta}.
\end{equation}
We note that~$w_\star>0$,
whence~$f_\star>0$ and~$v_\star>0$. In this spirit, the equilibrium~$E_0$ is ``trivial'',
since it corresponds to a case in which there is no vegetation, and correspondingly no
wildfires, and the hydric balance is the plain outcome of rain and evaporation.
Instead, the equilibrium~$E_1$ presents a ``nontrivial'' coexistence of water, vegetation, and  fire activity.

The linear stability\footnote{As customary for systems of ordinary differential equations,
we say that an equilibrium is:
\begin{itemize}\item
{\em linearly stable} if all eigenvalues 
of the linearized system have negative real parts,
\item {\em
linearly unstable}, if at least one eigenvalue has positive real part,
\item {\em neutral} if
at least one eigenvalue has null real part,
with none having positive real part.
\end{itemize}
For linearly stable equilibria,
perturbations decay, while for linearly unstable equilibria perturbations grow.
For neutral equilibria, the behavior of perturbations typically depends on higher-order terms.

On a similar note, for stable equilibria
the system is locally asymptotically stable (if one starts sufficiently close,
one stays close and converges to the equilibrium as time goes to infinity)
and for unstable equilibria
the system is locally asymptotically unstable (one can start arbitrarily close to the equilibrium and still move away). For neutral equilibria
higher-order nonlinear terms may either stabilize or destabilize the system, hence the linearization is inconclusive for
asymptotic statements.} of these equilibria is addressed by the following result:

\begin{theorem}\label{STAL}
The equilibrium~$E_0$ is unstable
(the corresponding linearized system
possesses
one real, positive eigenvalue
and two real, negative eigenvalues).

Let also
$$\Upsilon:=\frac{ 2\beta \gamma \left(\delta-\alpha\right) }{ \sqrt{\alpha^2 \epsilon^2 + 4 \alpha \beta \gamma \delta} + \alpha \epsilon} + \epsilon.$$

If~$\Upsilon>0$, the equilibrium~$E_1$ is stable (all the eigenvalues of the linearized system
have negative real parts).

If~$\Upsilon<0$, the equilibrium~$E_1$ is unstable
(the linearized system
possesses
one real, negative eigenvalue and two complex eigenvalues with positive real part).

If~$\Upsilon=0$, the equilibrium~$E_1$ is neutral (the linearized system
possesses
one real, negative eigenvalue
and two purely imaginary eigenvalues).
\end{theorem}

We observe that all the cases in Theorem~\ref{STAL} are possible, i.e.~$\Upsilon$
changes sign in dependence of different choices of the structural parameters,
since
$$\lim_{\alpha\to+\infty}\frac{ 2\beta \gamma \left(\delta-\alpha\right) }{ \sqrt{\alpha^2 \epsilon^2 + 4 \alpha \beta \gamma \delta} + \alpha \epsilon}=-\frac{\beta\gamma}\epsilon$$
and
$$\lim_{\alpha\to0}\frac{ 2\beta \gamma \left(\delta-\alpha\right) }{ \sqrt{\alpha^2 \epsilon^2 + 4 \alpha \beta \gamma \delta} + \alpha \epsilon}=+\infty.$$

The change of stability of the equilibrium~$E_1$ can be visualized with the aid
of Figures~\ref{FIG100} and~\ref{FIG101}. Indeed, Figure~\ref{FIG100} sketches a case in which~$\Upsilon>0$ and stability can be observed since trajectories spiral inward towards the
equilibrium. Instead, Figure~\ref{FIG101} presents a case in which~$\Upsilon<0$ and instability
is showcased by the outward direction of the spirals, which is confirmed by the enlargement of the corresponding frame as the termination time increases.

{F}rom the ecological point of view it is interesting to observe that the instability
of the equilibrium~$E_1$, corresponding to the case~$\Upsilon<0$, is triggered by large values of~$\alpha$, 
which in turn represents the fire's propensity to ignite: this confirms
the experience that bushfire managing is more critical
in the presence of highly flammable fuel.

\begin{figure}[h]
    \centering
    \includegraphics[width=0.35\textwidth]{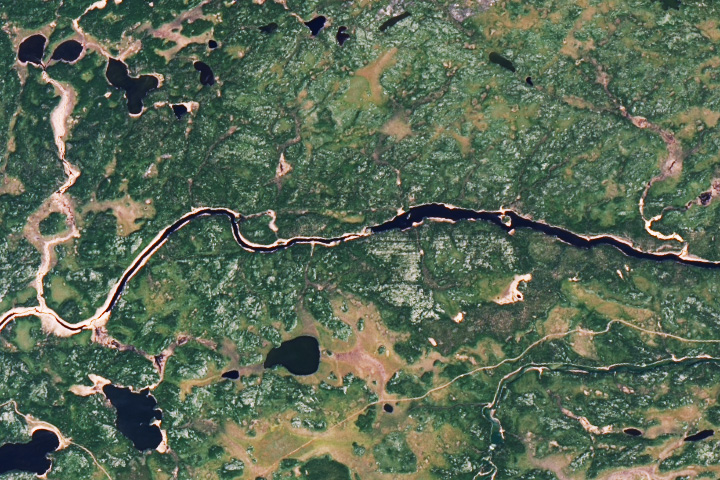}$\,$
    \includegraphics[width=0.35\textwidth]{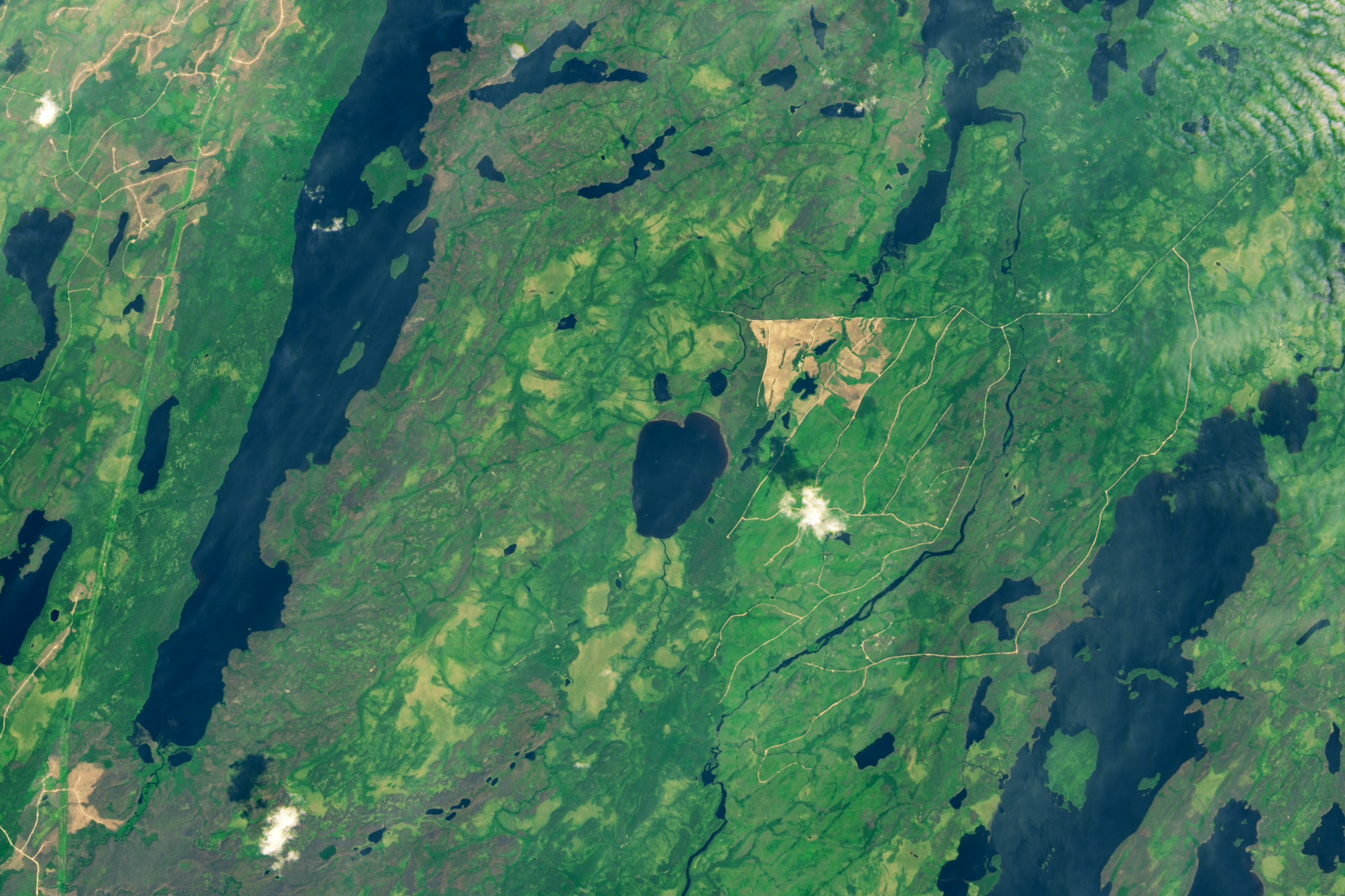}
    \caption{\footnotesize\sl Central Manitoba and Quebec, Canada: images captured by
the Operational Land Imager-2 (OLI-2 ) on NASA's satellite Landsat 9 (May 23, 2025, and
July 14, 2025) (Public Domain).}
    \label{FIGWA}
\end{figure}

We now return to the analysis of the system of partial differential
equations~\eqref{0dojc320yjyouj-3erfg}, with the aim of detecting both
the formation of patterns created by diffusion and the stabilization of the nontrivial
equilibrium~$E_1$ (for all values of the structural parameters)
thanks to high-frequency diffusive patterns.

With this objective in mind, we consider small amplitude perturbations of~$E_1$
corresponding to a single excited mode. Namely, given~$k\in(0,+\infty)^n$ and a small~$\rho>0$, we look for approximate
solutions of~\eqref{0dojc320yjyouj-3erfg} of the form
\begin{equation}\label{0dojc320yjyouj-3erfg2} \begin{dcases} f(x,t)=f_\star+\rho\Big(f_{1,k}(t)\,\sin(k\cdot x)+f_{2,k}(t)\,\cos(k\cdot x)\Big),\\
v(x,t)=v_\star+\rho\Big(v_{1,k}(t)\,\sin(k\cdot x)+v_{2,k}(t)\,\cos(k\cdot x)\Big),\\
w(x,t)=w_\star+\rho\Big(w_{1,k}(t)\,\sin(k\cdot x)+w_{2,k}(t)\,\cos(k\cdot x)\Big),
\end{dcases}
\end{equation}
where the functions~$f_{1,k}$, $f_{2,k}$, $v_{1,k}$, $v_{2,k}$,
$w_{1,k}$, and~$w_{2,k}$ are to be determined
so to obtain a complete solution of the linearized system.

In this setting, modes with high spatial frequency lead to stability
of the homogeneous equilibrium, according to the following result:

\begin{theorem}\label{TUT}
There exists~$k_0>0$ such that the equilibrium~$E_1$
becomes stable for~\eqref{0dojc320yjyouj-3erfg2} when~$|k|\ge k_0$.
\end{theorem}

\begin{figure}[h]
    \centering
    \includegraphics[height=3.1cm]{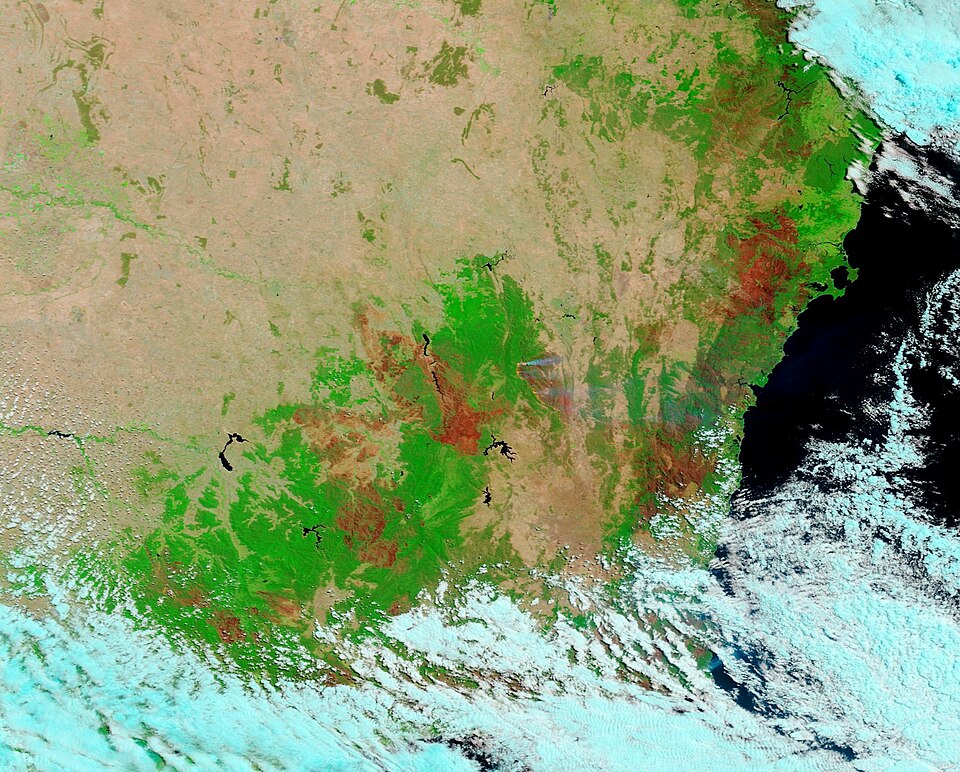}$\,$
    \includegraphics[height=3.1cm]{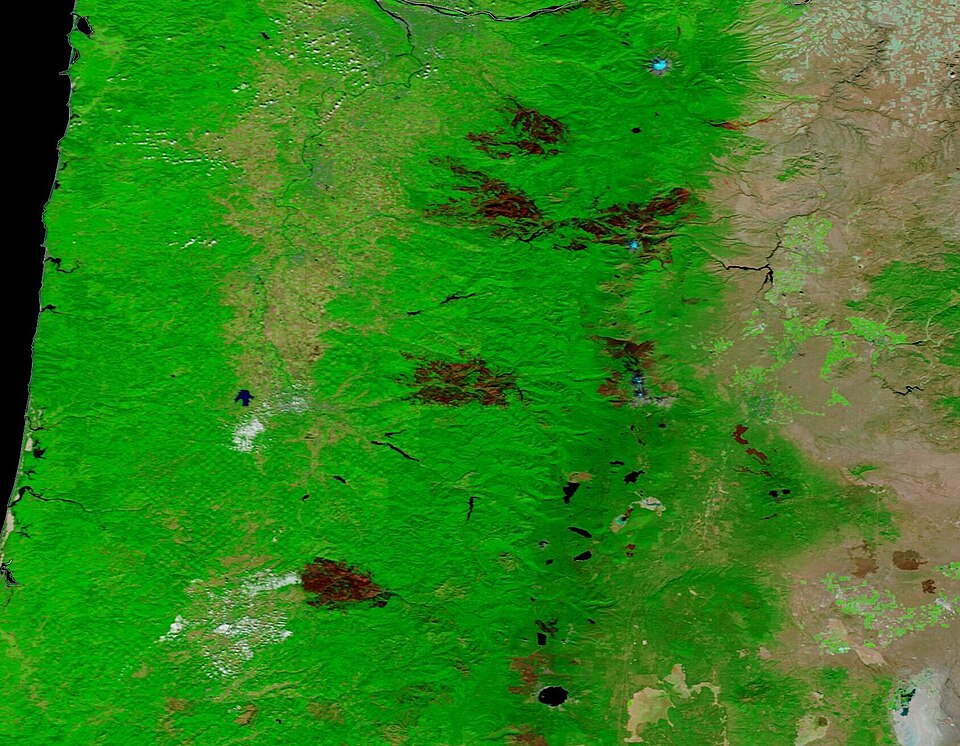}$\,$
    \includegraphics[height=3.1cm]{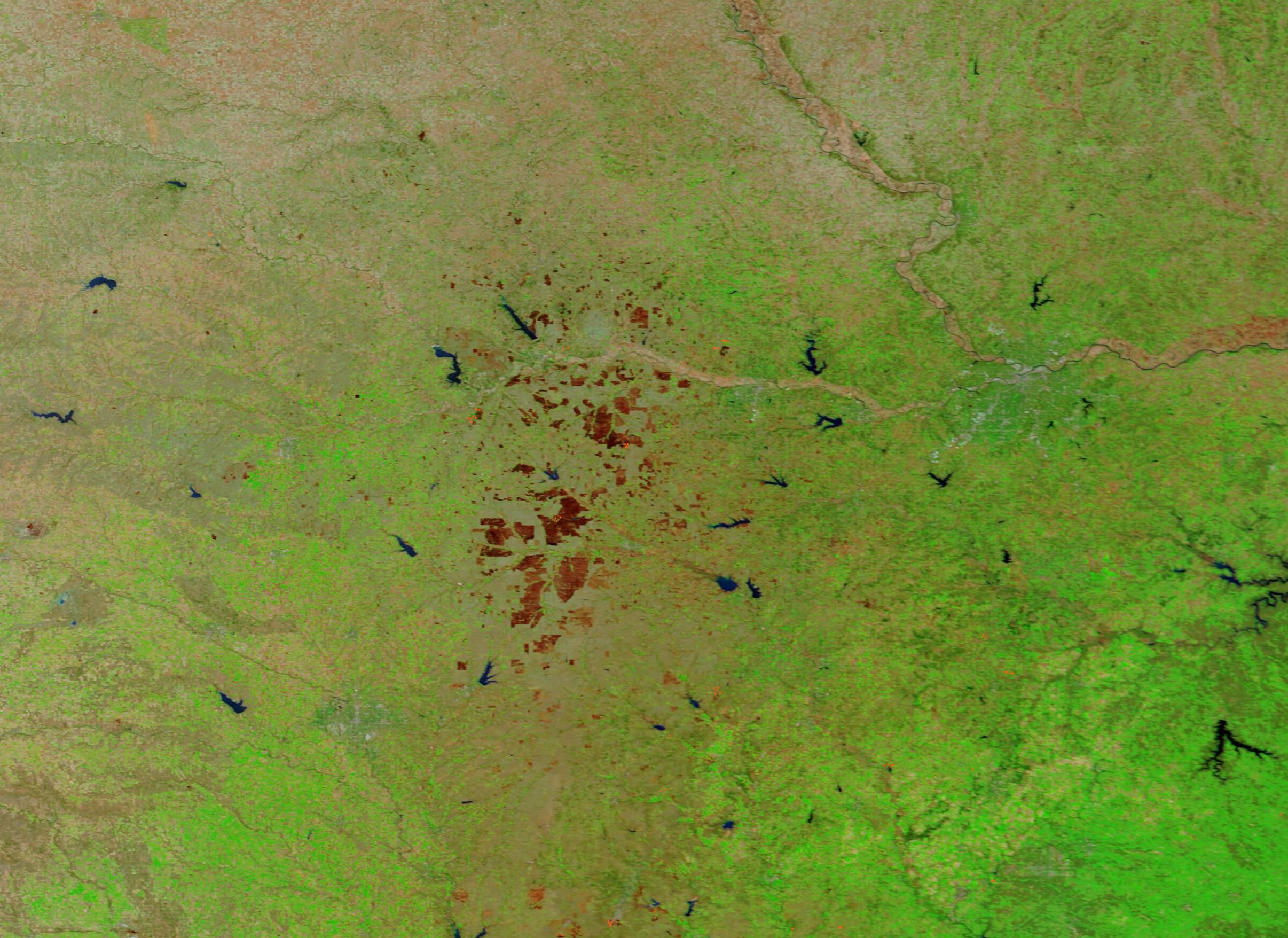}$\,$
    \includegraphics[height=3.1cm]{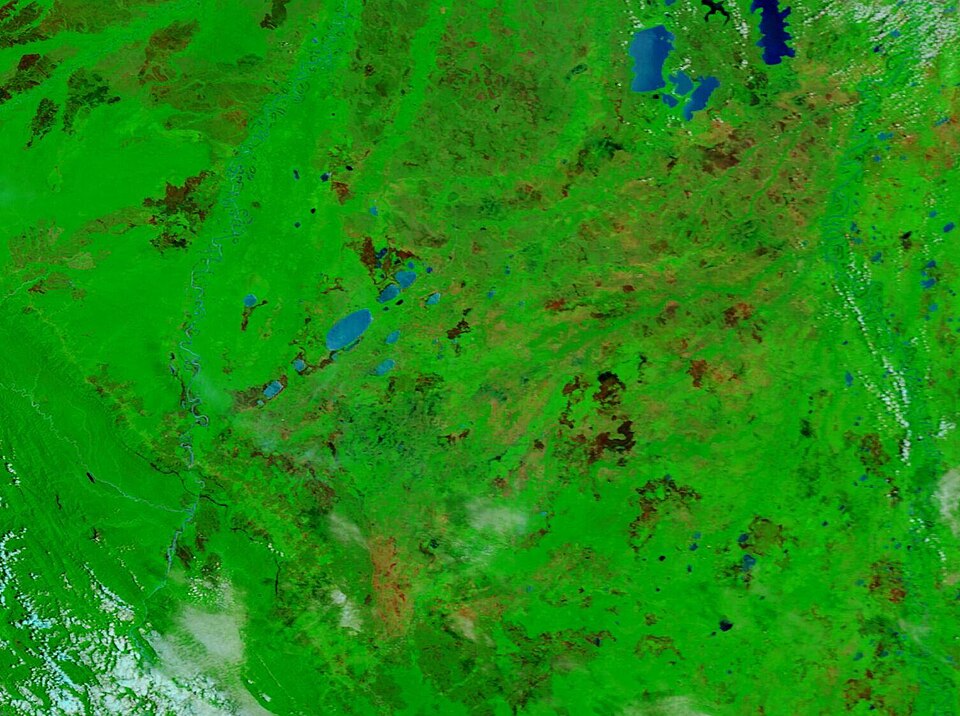}
    \caption{\footnotesize\sl False-color images acquired by the Moderate Resolution Imaging Spectroradiometer (MODIS) on board NASA's Terra satellite:
    burn scars in Australia (February 3, 2020),
    Oregon (September 27, 2020), 
    Kansas (April 12, 2023), and
    Bolivia (December 9, 2023) (Public Domain).}
    \label{FIG3}
\end{figure}

More specifically, Theorem~\ref{TUT} detects a change in the stability of the equilibrium~$E_1$ (which, according to Theorem~\ref{STAL}, may become unstable when~$\Upsilon \le 0$). In the case of a sufficiently high frequency, this equilibrium is driven toward a stable configuration (corresponding to the case~$\Upsilon > 0$ in Theorem~\ref{STAL}).
We stress that this does not provide a complete global analysis of all solutions to~\eqref{0dojc320yjyouj-3erfg}, since the superposition of different frequencies (and possibly infinitely many) is not covered by Theorem~\ref{TUT}. A full characterization of the asymptotic properties of such solutions in greater generality would be an interesting direction for future research.

We observe that the result showcased by Theorem~\ref{TUT} exhibits several structural differences compared to ``classical'' cases of Turing patterns. Indeed, the traditional literature\footnote{We also recall that
in the classical Turing instability only a finite number
of eigenmodes are unstable, because diffusion strongly damps high-frequency modes.
The situation described in Theorem~\ref{TUT} is instead structurally different.}
about Turing structures primarily focuses on reaction-diffusion systems involving two components: a self-sustaining activator and an inhibitor. In contrast, in our case, none of the quantities~$f$, $v$, or~$w$ is self-sustaining, as the linear terms in~\eqref{ODE} are either negative or sign-changing. Moreover, classical models typically assume a separation of diffusion scales, with a change of stability arising from the faster diffusion of the inhibitor catching up with the short-term self-enhancement of the activator. In our setting, however, the change of stability occurs independently of the ordering of the diffusion coefficients, and, in fact, it persists even when the diffusion rates of~$f$ and~$w$ are equal.

In this spirit, the model introduced here
is not just considering vegetation as an activator (since it grows locally and promotes more growth nearby)
with fire acting as an inhibitor (spreading over longer distances and suppressing vegetation): instead our model,
which is new even in the absence of diffusion, aims at capturing a more complex dynamic interplay
between wildfire activity, vegetation, and rainfall, with the advantage of allowing broader interactions
in the ecosystem and of posing no restrictions among different kind of diffusions (this is useful in practice,
because, for example, one can consider both the cases of fast
fire invasion taking place in bushfire events and
the slow wildfire diffusion induced by prescribed burning).

In addition, the classical Turing analysis aims at detecting
instability as an outcome of diffusion: instead,
in Theorem~\ref{TUT}, a suitable kind of diffusion
(namely the one related to small oscillations with high frequency)
is proven to ensure the stability of the system, even in situations
when all the homogeneous equilibria were unstable. 

The next result shows that a bifurcation from instability also takes place as an outcome of diffusion, namely, when~$\Upsilon<0$ and for a suitable choice of the diffusion coefficients, the equilibrium point~$E_1$ loses instability and a periodic cycle appears as a solution of the linearized equation, with a stable dynamics
of monocromatic waves. In a nutshell, given~$k\in(0,+\infty)^n$ and a small~$\rho>0$, we look for approximate
solutions of~\eqref{0dojc320yjyouj-3erfg} of the form
\begin{equation}\label{odjfcvfobi0efg-ortnghoP} \begin{dcases} f(x,t)=f_\star+\rho F ( k\cdot x+\sigma t),\\
v(x,t)=v_\star+\rho V ( k\cdot x+\sigma t),\\
w(x,t)=w_\star+\rho W ( k\cdot x+\sigma t),
\end{dcases}
\end{equation}
This setting corresponds to
``wave trains'', namely periodic waves whose front propagates with constant speed
and our objective is to determine~$k\in(0,+\infty)^n$, $\sigma\in(0,+\infty)$,
and nontrivial periodic functions $F$,
$V$, $W:\R\to\R$
so to obtain a complete solution of the linearized system.

\begin{theorem}\label{ojdfc30-t7ujsbdnfcvb}
Assume that~$\Upsilon<0$. 
Then, there exist~$k\in(0,+\infty)^n$ and~$\sigma\in(0,+\infty)$
for which~\eqref{odjfcvfobi0efg-ortnghoP} solves the linearized system for all~$\rho>0$.

More precisely, the functions~$F$, $V$, and~$W$ in~\eqref{odjfcvfobi0efg-ortnghoP} can be taken of the form
of ``monocromatic'' waves 
$$F(y)=\Re\big(F_0 e^{iy}\big),\qquad
V(y)=\Re\big(V_0 e^{iy}\big),\qquad W(y)=\Re\big(W_0 e^{iy}\big),$$
for a suitable~$(F_0, V_0, W_0)\in\C^3\setminus\{0\}$, and thus the corresponding periodic orbit lies in the linear space spanned by
the two vectors
$$ \big( \Re F_0,\Re V_0,\Re W_0\big)\qquad{\mbox{and}}\qquad
 \big( \Im F_0,\Im V_0,\Im W_0\big).$$

Furthermore, these solutions are exponentially attractive
for monocromatic waves, in the sense that solutions of the form~$Y(x,t):=\theta(t)\,e^{ik\cdot x}$, with~$\theta:\R\to\C^3$,
converge to superpositions of pure traveling trains~$e^{i(k\cdot x+\sigma t)}$ and~$e^{-i(\sigma t-k\cdot x)}$
exponentially fast in time.
\end{theorem}

We stress that 
Theorem~\ref{ojdfc30-t7ujsbdnfcvb} focuses on a linearized system. It would be desirable
to further investigate the nonlinear system in its full generality.

We also point out that Theorem~\ref{ojdfc30-t7ujsbdnfcvb} detects a phenomenon distinct from the classical Hopf instability, since the periodic orbits identified do not arise from a ``collision'' of eigenvalues. Instead, they result from a nonsingular crossing of two eigenvalues through the imaginary axis, while the third eigenvalue remains real and negative.
See Figure~\ref{FIG1qwefgb00} for a sketch of the eigenvalue evolution in terms of the diffusion coefficients (notice also that, even in this case, it is not necessary to assume that the diffusion speeds of firelines and water pools are different or ordered).

We also recall that
traveling wave patterns have been observed in nature in multiple contexts involving vegetation and fire.
See e.g.~\cite{Sprugel1976}
for the analysis of
cyclical band-shaped patterns 
of forest regeneration in balsam fir 
that migrate over time (this phenomenon is
called the ``fir wave'' in jargon).
See also~\cite{GANIE21, BANA23}
for models related to 
vegetation patterns in semi-arid ecosystems and
tree-grass interactions in fire-prone savannas, as well as~\cite{CART8905t2} for 
models describing interactions between vegetation and water 
including autotoxicity.
In concrete situations, the movement of traveling pulses can also be favored by sloped terrains, see e.g.~\cite{CAR89u0g36}.

As a final remark, let us point out that in the formulation of our model water, vegetation, and fire activity can be present simultaneously within the same spatial unit. This co-occurrence should be interpreted in an aggregated sense, reflecting temporal averaging over the model’s integration period rather than instantaneous coexistence at a single moment in time. This approach captures long-term tendencies and interactions among these components, but instantaneous physical incompatibilities (e.g., active fire and high water cover at the same exact location and time) are smoothed by the temporal averaging inherent in the model structure.

However, the presence of strong interactions between the different components of the model
formally corresponds to~$\alpha:=\beta:=\delta:=\eta:=\zeta:=+\infty$ in~\eqref{ODE},
reducing the model to
\begin{equation}\label{ODEREDS}
\begin{dcases}
f( v- w)=0, \\
v( w- f)=0,\\
v w=0.
\end{dcases}
\end{equation}
It is readily seen that the solutions of~\eqref{ODEREDS} prescribe at least two of the three components to vanish, namely, they are of the form~$(f,0,0)$, $(0,v,0)$, or~$(0,0,w)$, which represent ``segregated'' states in which only one of fire, vegetation, or water occurs at a given spatial location. In this sense, our model recovers such strict component separation as a limiting case.

As anticipated on page~\pageref{ANTIP6}, we now consider
a variation of the model in~\eqref{0dojc320yjyouj-3erfg} to account for competition among plants. A well-established formalization of this effect is implemented via kernel-based vegetation models, in which a kernel function describes how mutual influence decays with distance, see~\cite[equation~(25)]{LEFD} and the references therein.

In this spirit, one can model this nonlocal competition mechanisms at a spatial point~$x$ as a negative correction to vegetation growth of the form
\begin{equation}\label{PROVV2} \int_{\R^n} K(y)\,v(x+y)\,dy\end{equation}
and thus the evolution equation of the vegetation at a point~$x$ becomes
\begin{equation*}\partial_t v=v(\zeta w-\eta f)-v(x)\int_{\R^n} K(y)\,v(x+y)\,dy.\end{equation*}
The nonnegative kernel function~$K$ can be taken to be radial, to quantify the strength of competition at a given spatial lag.
For short-range inhibition, $K$ is sharply peaked near the origin and decays quickly away.
In this setting, one can approximate the kernel interaction in~\eqref{PROVV2}
by a higher-order differential operator of the form\footnote{As usual, we are using here the notation
of iterated Laplacian
$$\Delta^j:=\underbrace{\Delta\circ\cdots\circ\Delta}_{\text{\tiny{$j$ times}}}.$$}
\begin{equation}\label{NUOVO7}
\ell_0 v+\ell_1\,{\Delta} v+\ell_2\,\Delta^2 v+\ell_3\,{\Delta}^3 v+\dots,\end{equation}
for some~$\ell_j\in(0,+\infty)$, see Appendix~\ref{0apjclrgh}.

In the interest of simplicity, in this paper we will keep only the term related to the Laplacian operator (e.g., with the ansatz that the first term is compensated by short-range pollination and the higher-order terms are somewhat negligible). Namely, we replace~\eqref{NUOVO7} by the only term~$\ell{\Delta} v$, for some~$\ell\in(0,+\infty)$.

Accordingly, the evolution equation of the vegetation becomes\footnote{We point out that terms of the form~$v\Delta v$,
or of the form~$\Delta (v^2)$, are related to self-diffusion 
and porous medium equations, 
and can also lead to
pattern formation in other contexts, see~\cite{LOUSELF}.}
$$ \partial_t v=v(\zeta w-\eta f)-\ell v{\Delta} v$$
and correspondingly one replaces~\eqref{0dojc320yjyouj-3erfg} by
\begin{equation}\label{NUOVO}
\begin{dcases}
\partial_t f=f(\alpha v-\beta w)+c\Delta f, \\
\partial_t v=v(\zeta w-\eta f)-\ell v{\Delta} v,\\
\partial_t w=\gamma-\delta v w-\epsilon w+d\Delta w.
\end{dcases}
\end{equation}

The vegetation competition term promotes further pattern formation, which stems from an
induced instability of the homogeneous equilibrium.
This holds for short-frequency oscillations, corresponding to small~$k$ in~\eqref{0dojc320yjyouj-3erfg2},
and small rainfall amount~$\gamma$, as shown by the following result:

\begin{theorem}\label{TR5r3ee}
Let~$\varsigma\in(0,\epsilon)$.
There exist~$k_0$, $\gamma_0\in(0,+\infty)$ such that
the equilibrium~$E_1$ for~\eqref{NUOVO}
becomes unstable for~\eqref{0dojc320yjyouj-3erfg2} 
when~$k\in(0,k_0)$, $\gamma\in(0,\gamma_0)$, and~$\ell:=\frac{\varsigma}{|k|^2 v_\star}$.
\end{theorem}

We emphasize that Theorem~\ref{TR5r3ee} concerns only the stability analysis of the linearized system. A more
comprehensive treatment of the nonlinear system in its full generality remains the subject of future work.

The role of plant competition in the stability of the homogeneous equilibrium
is thus explicit by comparing Theorem~\ref{TUT} with Theorem~\ref{TR5r3ee}. We also stress that, in the absence of reciprocal negative interaction among neighboring plants, the small rainfall regime
corresponds to a stable homogeneous equilibrium~$E_1$
(because~$\Upsilon$ approaches~$\epsilon>0$ for small~$\gamma$): in this sense Theorem~\ref{TR5r3ee} genuinely
detects a change of stability which is specifically induced
by plant competition.

The rest of this paper is organized as follows. Section~\ref{SEC:II}
analyzes the system of ordinary differential equations~\eqref{ODE}
and presents the proof of Theorem~\ref{STAL}.

The analysis of the system of partial differential equations~\eqref{0dojc320yjyouj-3erfg}
is carried out in Sections~\ref{SEC:III} and~\ref{SEC:IIII}, together with the proofs of
Theorems~\ref{TUT} and~\ref{ojdfc30-t7ujsbdnfcvb}.

The system accounting for competition in vegetation is studied in Section~\ref{TR5r3eeS}, where we present
the proof of Theorem~\ref{TR5r3ee}.

Appendix~\ref{APP;a} contains an auxiliary observation regarding the roots of cubic polynomials 
and Appendix~\ref{0apjclrgh} discusses the asymptotics linking
the nonlocal kernel interaction in~\eqref{PROVV2} to the differential operator in~\eqref{NUOVO7}.

\section{Linearization, stability analysis of ordinary differential equations, and proof of Theorem~\ref{STAL}}\label{SEC:II}

The Jacobian matrix of~\eqref{ODE} is
\[
J(f,v,w) =
\begin{bmatrix}
\alpha v - \beta w  & \alpha f & -\beta f \\
- \eta v & \zeta w - \eta f & \zeta v\\
0 & -\delta w & -\delta v - \epsilon
\end{bmatrix}.
\]
Thus,
\[
J (E_0)=
\begin{bmatrix}
- \frac{\beta \gamma}{\epsilon} & 0 & 0 \\
0 & \frac{\gamma\zeta}{\epsilon}&0\\
0 & - \frac{\delta \gamma}{\epsilon}& -\epsilon 
\end{bmatrix}
\]
and
\begin{equation}\label{JE1}
J(E_1) =
\begin{bmatrix}
0 & \alpha f_\star & -\beta f_\star\\
-\eta v_\star & 0 & \zeta v_\star\\
0 & -\delta w_\star & -\delta v_\star - \epsilon
\end{bmatrix}.
\end{equation}

The eigenvalues of~$J(E_0)$ are
$$ -\frac{\beta \gamma}{\epsilon},\qquad
\frac{\gamma\zeta}{\epsilon},\qquad-\epsilon.$$
Since one eigenvalue is positive
(and two are negative)
the equilibrium~$E_0$ is unstable.

Also, the characteristic polynomial of~$J(E_1)$ is
\[P(\lambda):=
\lambda^3 + a_2 \lambda^2 + a_1 \lambda + a_0 ,
\]
where
\begin{equation}\label{0jdocwg-21rtghb:001}
\begin{split}
a_2 &:= \delta v_\star + \epsilon, \\
a_1 &=  \delta\zeta v_\star w_\star + \alpha \eta f_\star v_\star, \\ {\mbox{and }}\quad
a_0 &= \alpha \eta f_\star v_\star (\delta v_\star + \epsilon) + \beta \delta \eta f_\star v_\star w_\star.
\end{split}\end{equation}

We stress that~$a_0>0$, $a_1>0$, and~$a_2>0$. Moreover,
\begin{equation}\label{0jdocwg-21rtghb:002}
\begin{split}&
\frac{a_1 a_2-a_0}{\delta\zeta v_\star w_\star}=
\delta  v_\star  -\frac{\beta \eta f_\star}\zeta   + \epsilon=
\left( \frac{ \delta}{\alpha}  -1\right)\beta w_\star + \epsilon\\&\qquad=
\frac{ \left(\delta-\alpha\right) \left(\sqrt{\alpha^2 \epsilon^2 + 4 \alpha \beta \gamma \delta} - \alpha \epsilon\right) }{2 \alpha\delta} + \epsilon\\&\qquad=
\frac{ 2\beta \gamma \left(\delta-\alpha\right) }{ \sqrt{\alpha^2 \epsilon^2 + 4 \alpha \beta \gamma \delta} + \alpha \epsilon} + \epsilon\\&\qquad=\Upsilon.
\end{split}\end{equation}
{F}rom this and Lemmata~\ref{LE:A}, \ref{LEMMA2}, and~\ref{APK0-2CMAswqpr}, one readily concludes the proof of Theorem~\ref{STAL}.
More specifically, if~$\Upsilon>0$ we apply
Lemma~\ref{LE:A}, obtaining that all eigenvalues have negative real parts.
If instead~$\Upsilon=0$ we apply
Lemma~\ref{LEMMA2}, finding one real, negative eigenvalue
and two purely imaginary eigenvalues.
Finally, if~$\Upsilon<0$, we recall Lemma~\ref{APK0-2CMAswqpr} and we get
one real, negative eigenvalue and two complex eigenvalues with positive real part.

\section{Change of stability and proof of Theorem~\ref{TUT}}\label{SEC:III}

Now we analyze the mechanism for which, in the presence of diffusion, the possibly unstable homogeneous equilibrium~$E_1$ can become stable with respect to spatially nonuniform perturbations
of high frequency.

In order to achieve this goal, we plug~\eqref{0dojc320yjyouj-3erfg2}
into~\eqref{0dojc320yjyouj-3erfg} and disregard quadratic terms in~$\rho$.
With this approximation, 
and using the short notation~$\mu:=|k|^2$,
one thus obtains
%
%
two decoupled and identical
linear systems of ordinary differential equations.
Namely, one can set either~$X:=\big( f_{1,k},v_{1,k},w_{1,k}\big)$ or~$X:=\big( 
f_{2,k},v_{2,k},w_{2,k}\big)$ and write the problem in the compact form~$\dot X= AX$, where
\begin{equation}\label{AMATRIX} A:=\begin{bmatrix}
-c\mu & \alpha f_\star & -\beta f_\star \\
-\eta v_\star &0&\zeta v_\star\\
0&-\delta w_\star &\tiny{\mbox{$-\delta v_\star -\epsilon -d\mu$}}
\end{bmatrix}.\end{equation}

The characteristic polynomial of~$A$ takes the form
\begin{equation}\label{coeffA0987:CPSD} P(\lambda)=\lambda^3+ a_2(\mu)\,\lambda^2+a_1(\mu)\,\lambda+a_0(\mu),\end{equation}
with
\begin{equation}\label{coeffA0987}\begin{split}
a_2(\mu)&:= (c+d)\mu + \delta v_\star + \epsilon , \\
a_1(\mu) &:=c\mu(d\mu+\delta v_\star + \epsilon ) + \alpha\eta f_\star v_\star + \delta\zeta v_\star  w_\star, \\
{\mbox{and }}\qquad
a_0 (\mu)&:=c\mu \delta\zeta v_\star w_\star + \alpha \eta f_\star v_\star (\delta v_\star + \epsilon + d\mu) + \beta\delta \eta
f_\star  v_\star w_\star.
\end{split}\end{equation}
We stress that
\begin{equation}\label{coeffA0987:pos}
a_2(\mu),\;a_1(\mu),\;a_0(\mu)\in(0,+\infty).
\end{equation}

We also observe that, in the absence of diffusion,
i.e.~$c:=d:=0$, the matrix~$A$ reduces to the matrix in~\eqref{JE1}, and the same holds true when~$\mu:=0$.
As a result, recalling~\eqref{0jdocwg-21rtghb:002},
\begin{equation*}
a_1(0)\, a_2(0)-a_0(0)=\delta\zeta v_\star w_\star\Upsilon.\end{equation*}

More generally,
\begin{equation}\label{bcoff6479832iPRE}
a_1(\mu)\, a_2(\mu) - a_0(\mu)= b_3\,\mu^3
+ b_2\,\mu^2 + b_1\,\mu +b_0,
\end{equation}
where
\begin{equation}\label{bcoff6479832i}\begin{split}&
b_3:=cd(c+d),\\&
b_2:=c(c+2d)(\delta v_\star + \epsilon),\\& 
b_1:=c(\delta v_\star + \epsilon)^2 + d\delta\zeta v_\star w_\star + c\alpha\eta f_\star v_\star ,\\
{\mbox{and }}\qquad&
b_0:=\delta v_\star w_\star
\big( \delta\zeta v_\star + \zeta\epsilon - \beta\eta f_\star \big)=
\delta\zeta v_\star w_\star\Upsilon.\end{split}
\end{equation}
We remark that
\begin{equation}\label{coeffA0987:pos12}
b_3,\;b_2,\;b_1\in(0,+\infty).
\end{equation}
We also recall that, by~\eqref{coeffA0987:pos} and Lemma~\ref{LE:A}, stability holds true anytime~$a_1(\mu)\, a_2(\mu) - a_0(\mu)>0$. This and~\eqref{coeffA0987:pos12} yield stability
whenever~$\Upsilon\ge0$ (hence, in this case, diffusion does not disrupt stability).

Furthermore, by~\eqref{bcoff6479832iPRE} and~\eqref{coeffA0987:pos12},
$$\lim_{\mu\to+\infty}a_1(\mu)\, a_2(\mu) - a_0(\mu)=+\infty,$$
from which we conclude that~$a_1(\mu)\, a_2(\mu) - a_0(\mu)>0$ as long as~$\mu$ is large enough,
thus concluding the proof of
Theorem~\ref{TUT}.

\section{Periodic oscillations, traveling waves, and proof of Theorem~\ref{ojdfc30-t7ujsbdnfcvb}}\label{SEC:IIII}

One plugs~\eqref{odjfcvfobi0efg-ortnghoP} into~\eqref{0dojc320yjyouj-3erfg} finding that,
up to quadratic orders in~$\rho$ and using the short notation~$y:=k\cdot x+\sigma t\in\R$,
\begin{equation}\label{sq-dpcjvrtwpho-ytDj} \begin{dcases} \sigma
F'(y)=f_\star\big(\alpha V(y)-\beta W(y)\big)+c|k|^2 F''(y),\\ \sigma
V'(y)=v_\star\big(\zeta W(y)-\eta F(y)\big),\\ \sigma
W'(y)=-\delta v_\star W(y)-\delta w_\star V(y)-\epsilon W(y)+d|k|^2 W''(y).
\end{dcases}
\end{equation}
We look for solutions in the form
$$F(y)=F_0 e^{iy},\qquad
V(y)=V_0 e^{iy},\qquad W(y)=W_0 e^{iy},$$
with~$(F_0, V_0, W_0)\in\C^3\setminus\{0\}$ to be determined.

In this way, and setting~$\mu:=|k|^2$, we rewrite~\eqref{sq-dpcjvrtwpho-ytDj} in the form
\begin{equation}\label{sq-dpcjvrtwpho-ytDj.EQ} \begin{dcases} 
i\sigma F_0=f_\star\big(\alpha V_0-\beta W_0\big)-c\mu F_0,\\
i\sigma V_0=v_\star\big(\zeta W_0-\eta F_0\big),\\
i\sigma W_0=-\delta v_\star W_0-\delta w_\star V_0-\epsilon W_0-d\mu W_0.
\end{dcases}
\end{equation}
Hence, with the vector notation~$ X:=(F_0,V_0,W_0)$, we obtain the compact equation
\begin{equation}\label{a0dcjorgh0typj}
AX=i\sigma X,\end{equation} with~$A$ as in~\eqref{AMATRIX}.

\begin{figure}[h]
    \centering
    \includegraphics[height=5.9cm]{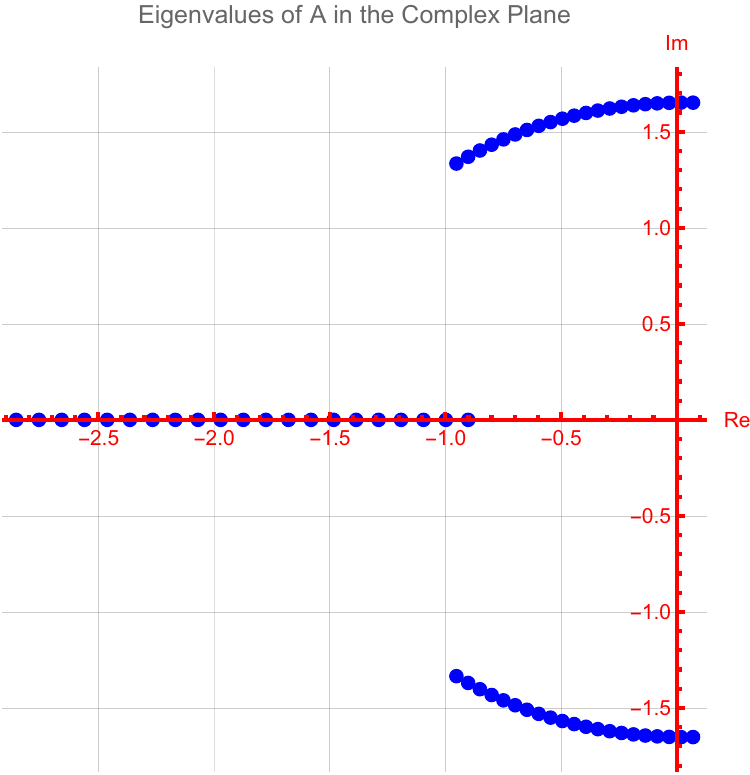}
    \caption{\footnotesize\sl Displacement of the eigenvalues of the matrix~$A$ in~\eqref{AMATRIX}
when~$\alpha :=2$, $ \beta := \gamma := \delta :=  \eta := \zeta := 1$, $\epsilon :=\frac1{10}$, and~$c|k|^2=d|k|^2\in[0,2]$.
The periodic solutions for the linearized system in Theorem~\ref{ojdfc30-t7ujsbdnfcvb} arise when a pair of eigenvalues crosses the imaginary axis from right to left as~$\mu$ increases. Note that the third eigenvalue remains real and negative.}
    \label{FIG1qwefgb00}
\end{figure}

The characteristic polynomial of~$A$ is
$$ P(\lambda)=\lambda^3+ a_2(\mu)\,\lambda^2+a_1(\mu)\,\lambda+a_0(\mu),$$
with coefficients~$a_2(\mu)$, $a_1(\mu)$ and~$a_0(\mu)$ as in~\eqref{coeffA0987}.

We stress that~$a_1(\mu)>0$. Furthermore, equation~\eqref{a0dcjorgh0typj} is solvable if and only if~$P$
admits a purely imaginary root.
By Lemma~\ref{LEMMA2}, we know that this is equivalent to $$a_1(\mu)\, a_2(\mu)-a_0(\mu)=0.$$ Namely, recalling~\eqref{bcoff6479832iPRE},
\begin{equation}\label{0jdocwg-21rtghb:003}
\Phi(\mu):=b_3\,\mu^3+ b_2\,\mu^2 + b_1\,\mu +b_0=0,
\end{equation}
where the coefficients~$b_0$, $b_1$, $b_2$ and~$b_3$ are given in~\eqref{bcoff6479832i}.

We also know by Lemma~\ref{LEMMA2} that if these conditions are met,
the roots of~$P$ are~$-a_2(\mu)$, $i\sqrt{a_1(\mu)}$, and~$-i\sqrt{a_1(\mu)}$,
hence~$\sigma =\sqrt{a_1(\mu)}$.

We point out that, when~$\Upsilon<0$,
$$\Phi(0)=b_0=\delta\zeta v_\star w_\star\Upsilon<0.$$
Also, since~$b_3>0$,
$$\lim_{\mu\to+\infty}\Phi(\mu)=+\infty$$
and so, by continuity, we find~$\mu_\star$ such that~$\Phi(\mu_\star)=0$,
which fulfills~\eqref{0jdocwg-21rtghb:003}, as advertised
(and notice that the choice~$\mu:=\mu_\star$ entails a choice~$\sigma:=\sigma_\star$).

The eigenvector~$X_\star$ of~$A$ corresponding to~$i\sigma_\star$ provides the desired solution of~\eqref{a0dcjorgh0typj}.
Also, $a_2(\mu_\star)>0$, yielding the stability of the equilibrium (recall that
the corresponding eigenvalues are~$-a_2(\mu_\star)$, $i\sqrt{a_1(\mu_\star)}$, and~$-i\sqrt{a_1(\mu_\star)}$,
due to Lemma~\ref{LEMMA2}).

In further detail, a (normalized) eigenvector~$X_\star=(F_\star,V_\star,W_\star)$ provides a complex solution of~\eqref{sq-dpcjvrtwpho-ytDj.EQ}
and therefore of~\eqref{sq-dpcjvrtwpho-ytDj}. Since the system in~\eqref{sq-dpcjvrtwpho-ytDj} is linear and with
real coefficients, the real (as well as the imaginary) part of this solution gives a real solution, which can thus be written as
$$\begin{dcases}&F(y)=\Re\big(F_\star e^{iy}\big)=\Re F_\star \cos y-\Im F_\star\sin y,\\ &
V(y)=\Re\big(V_\star e^{iy}\big)=\Re V_\star \cos y-\Im V_\star\sin y,\\& W(y)=\Re\big(W_\star e^{iy}\big)=\Re W_\star \cos y-\Im W_\star\sin y.\end{dcases}$$

In particular, for all~$y\in\R$ the vector~$(F(y),V(y),W(y))$ is spanned by the vectors
$$( \Re F_\star,\Re V_\star,\Re W_\star)\qquad{\mbox{and}}\qquad( \Im F_,\Im V_\star,\Im W_\star).$$

We now check the attractivity of the perturbed dynamic. To this end, we consider
solutions of the linearized system in the form~$Y(x,t):=\theta(t)\,e^{ik\cdot x}$ for suitable~$\theta:\R\to\C^3$. That is, $Y$ solves~$\partial_t Y=MY+\Lambda\Delta Y$, where
$$M:=\begin{bmatrix}
0 & \alpha f_\star & -\beta f_\star\\
-\eta v_\star &0&\zeta v_\star\\
0&-\delta w_\star &-\delta v_\star -\epsilon
\end{bmatrix}\qquad{\mbox{and}}\qquad
\Lambda:=\begin{bmatrix}c&0&0\\0&0&0\\0&0&d\end{bmatrix}.$$
Hence, separating sines and cosines, and noticing that~$M-|k|^2\Lambda=A$,
we find that~$\dot \theta=A\theta$
and consequently\begin{equation}\label{0ojd4390iy0AScdni}\theta(t)=e^{At} \theta(0).\end{equation}

Now we remark that the complex conjugate~$\overline{X_\star}$ of~$X_\star$ is an eigenvector for~$A$ corresponding to the eigenvalue~$-i\sigma$.
We can thus consider~$\C^3$ as the span of~$\{ X_\star, \overline{X_\star}, \omega_\star\}$, where~$\omega_\star$
is a (normalized) eigenvector with eigenvalue~$-a_2(\mu_\star)\in(-\infty,0)$.
Hence, we write~$\theta(0)=
\theta_1 X_\star+\theta_2 \overline{X_\star}+\theta_3 \omega_\star$, for some~$\theta_1$, $\theta_2$, $\theta_3\in\C$.

This allows one to reformulate~\eqref{0ojd4390iy0AScdni} as follows:
$$ \theta(t)=\theta_1\,e^{i\sigma t}X_\star+\theta_2\,e^{-i\sigma t}\overline{X_\star}+\theta_3\,e^{-a_2 t}\omega_\star
.$$
On this account, $$Y(x,t)=\theta_1\,e^{i(\sigma t+k\cdot x)}X_\star+\theta_2\,e^{-i(\sigma t-k\cdot x)} \overline{X_\star}+\theta_3\,e^{-a_2 t+ik\cdot x}\omega_\star,$$
showing the exponential decay in time towards pure traveling trains.

This completes the proof of Theorem~\ref{ojdfc30-t7ujsbdnfcvb}
(in that statement, the suffix~$\star$ has been dropped for notational consistence
with the rest of the Introduction).

\section{Competition for resource in vegetation and proof of Theorem~\ref{TR5r3ee}}\label{TR5r3eeS}

We observe that the presence of the competition term in~\eqref{NUOVO} modifies the vegetation components 
into
\begin{equation*}
\begin{dcases}
\dot v_{1,k}=v_\star(\zeta w_{1,k}-\eta f_{1,k} )+\ell |k|^2 v_\star v_{1,k},\\
\dot v_{2,k}=v_\star(\zeta w_{2,k}-\eta f_{2,k})+\ell |k|^2 v_\star  v_{2,k} .
\end{dcases}
\end{equation*}
As a result, using the notation~$\mu:=|k|^2$ and
recalling~\eqref{AMATRIX}, the problem can be cast into
the compact form~$\dot X= LX$, where
\begin{equation}\label{pkdweLL} L:=
A+\ell\mu v_\star
\begin{bmatrix}
0&0&0 \\0&1&0\\
0&0&0
\end{bmatrix}
=
\begin{bmatrix}
-c\mu & \alpha f_\star & -\beta f_\star \\
-\eta v_\star &\ell\mu v_\star&\zeta v_\star\\
0&-\delta w_\star &\tiny{\mbox{$-\delta v_\star -\epsilon -d\mu$}}
\end{bmatrix}.\end{equation}
In a nutshell, the matrix~$L$ is obtained from~$A$ by adding, say, a column (the second one)
and ditto for the matrix~$\lambda{\rm I}-L$ with respect to the matrix~$\lambda{\rm I}-A$.
We can thus use the fact that the determinant is ``multilinear'' (i.e., it is a linear function in each column of the input matrix) and find that
the characteristic polynomial of~$L$ has the form
\begin{equation}\label{qi0dwfuojvrbthwq8nc7rvtb9mA}\begin{split}&
Q(\lambda):=\det(\lambda{\rm I}-L)=
\det(\lambda{\rm I}-A)-\det
\begin{bmatrix}
\lambda+c\mu & 0 & \beta f_\star \\
\eta v_\star &\ell\mu v_\star&-\zeta v_\star\\
0&0 &\tiny{\mbox{$\lambda+\delta v_\star +\epsilon +d\mu$}}
\end{bmatrix}\\&\qquad\qquad=P(\lambda)-\ell\mu v_\star \det
\begin{bmatrix}
\lambda+c\mu & \beta f_\star \\
0 &\tiny{\mbox{$\lambda+\delta v_\star +\epsilon +d\mu$}}
\end{bmatrix}\\&\qquad\qquad=P(\lambda)-\ell\mu v_\star\Big(
\lambda^2 + \big( \delta v_\star + \epsilon + (c + d) \mu\big) \lambda
+c \mu ( \delta v_\star+ \epsilon + d \mu) \Big),\end{split}\end{equation}
where~$P$ is the characteristic polynomial of~$A$, as computed in~\eqref{coeffA0987:CPSD}
and~\eqref{coeffA0987}.

The identity in~\eqref{qi0dwfuojvrbthwq8nc7rvtb9mA} shows that~$Q$ differs from~$P$ in the constant, linear, and quadratic terms with
respect to~$\lambda$. On this account, in the setting of~\eqref{coeffA0987:CPSD}
and~\eqref{coeffA0987}, we write
\begin{equation}\label{i02ro48bmn87X0}
Q(\lambda)=\lambda^3+ \widehat{a}_2(\mu)\,\lambda^2+\widehat{a}_1(\mu)\,\lambda+\widehat{a}_0(\mu),\end{equation}
with
\begin{equation}\label{i02ro48bmn87X02}\begin{split}
&\widehat{a}_2(\mu):=a_2(\mu)-\ell\mu v_\star,\\
&\widehat{a}_1(\mu):=a_1(\mu)-\ell\mu v_\star \big( \delta v_\star + \epsilon + (c + d) \mu\big),\\
{\mbox{and }}\quad& \widehat{a}_0(\mu):=a_0(\mu)
-c\ell\mu^2 v_\star( \delta v_\star+ \epsilon + d \mu) .
\end{split}\end{equation}

In fact, we choose~$\ell:=\frac{\varsigma}{\mu v_\star}$, with~$\varsigma\in(0,\epsilon)$, and
we make the notation more explicit by writing~$\nu:=(\gamma,\mu)$, since we are interested
in bifurcating our parameters from~$\nu=0$. With this more precise notation, we rewrite~\eqref{i02ro48bmn87X0}
and~\eqref{i02ro48bmn87X02} as
\begin{equation}\label{i-dpbXJm0nXZ01}
Q(\lambda,\nu)=\lambda^3+ \widetilde{a}_2(\nu)\,\lambda^2+\widetilde{a}_1(\nu)\,\lambda+\widetilde{a}_0(\nu),\end{equation}
with
\begin{equation}\label{i-dpbXJm0nXZ0}\begin{split}
&\widetilde{a}_2(\nu):=a_2(\mu)-\varsigma,\\
&\widetilde{a}_1(\nu):=a_1(\mu)-\varsigma \big( \delta v_\star + \epsilon + (c + d) \mu\big),\\
{\mbox{and }}\quad& \widetilde{a}_0(\nu):=a_0(\mu)
-c\mu\varsigma( \delta v_\star+ \epsilon + d \mu) .
\end{split}\end{equation}

We stress that, in light of~\eqref{coeffA0987}, the coefficients~$a_2(\mu)$, $a_1(\mu)$, and~$a_1(\mu)$
also depend on~$\gamma$, but this dependence
only occurs via~$f_\star$, $v_\star$, and~$w_\star$ (which in turn do not depend on~$\mu$).

Moreover, in light of~\eqref{EQSTvdARDAF23-4},
$$ \lim_{\nu\to(0,0)}w_\star=
\lim_{\gamma\searrow0}w_\star=0$$
and therefore 
$$  \lim_{\nu\to(0,0)}f_\star=0\qquad{\mbox{and}}\qquad  \lim_{\nu\to(0,0)}v_\star=0.$$
For this reason, recalling~\eqref{coeffA0987}, we see that
\begin{equation*}\begin{split} \lim_{\nu\to(0,0)} a_2(\mu)= \epsilon
\qquad{\mbox{and}}\qquad \lim_{\nu\to(0,0)}a_1(\mu) =0= \lim_{\nu\to(0,0)}
a_0 (\mu).
\end{split}\end{equation*}

{F}rom this and~\eqref{i-dpbXJm0nXZ0} we arrive at
$$\widetilde{a}_2(0)=\epsilon-\varsigma,\qquad
\widetilde{a}_1(0)=-\varsigma \epsilon ,\qquad{\mbox{and }}\qquad
\widetilde{a}_0(0)=0.$$
Hence, by~\eqref{i-dpbXJm0nXZ01},
$$Q(\lambda,0)=\lambda^3+ (\epsilon-\varsigma)\lambda^2-\varsigma\epsilon\lambda=\lambda(\lambda-\varsigma)(\lambda+\epsilon).$$
As a result, we see that~$\varsigma$ is a positive, simple root of~$Q(\cdot,0)$ and therefore, by the Implicit Function Theorem,
for~$\nu$ sufficiently small there exists~$\lambda(\nu)$, continuously depending on~$\nu$, such that~$\lambda(0)=\varsigma$
and~$Q(\lambda(\nu),\nu)=0$. In this way, for small~$\nu$, the positive root~$\lambda(\nu)$
provides the desired instability and
the proof of Theorem~\ref{TR5r3ee}
is thereby completed.

See Figure~\ref{Fe2IG1qwefgb0215rnb78X.90} for a visualization
of the instability induced by plant competition

\begin{figure}[h]
    \centering
    \includegraphics[height=5.9cm]{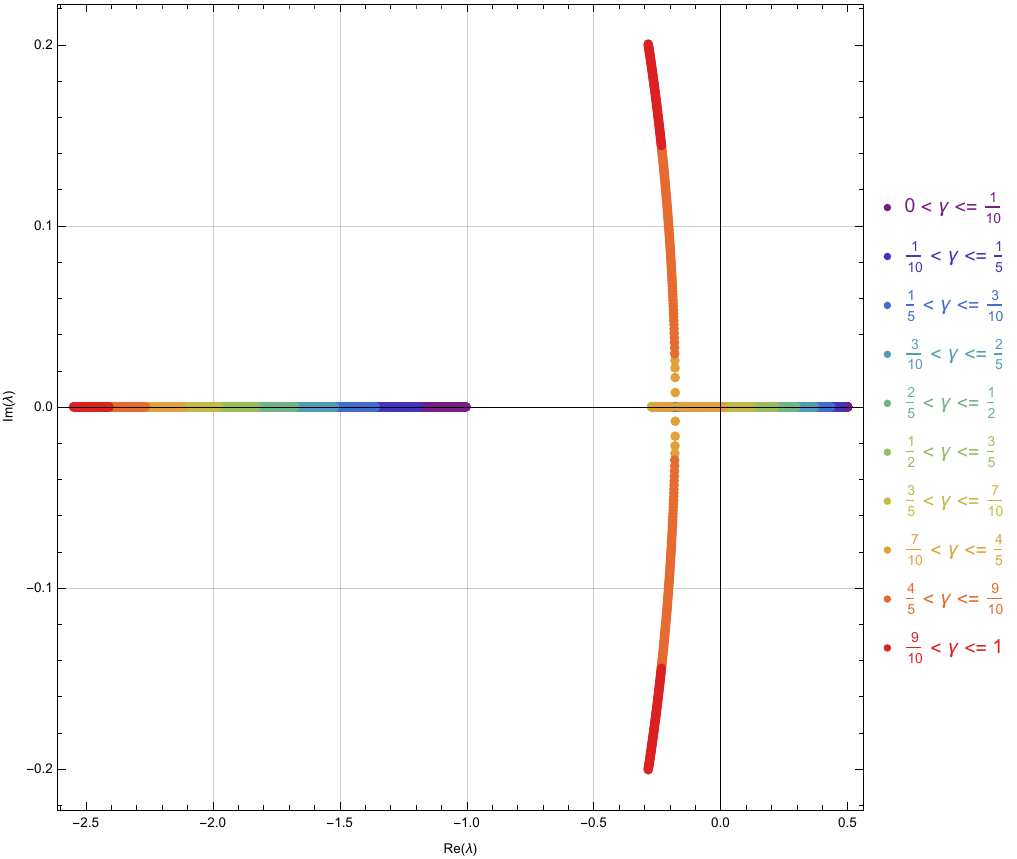}
    \caption{\footnotesize\sl Displacement of the eigenvalues of the matrix~$L$ in~\eqref{pkdweLL}
    when~$\alpha := \beta := \delta := \epsilon := \eta := \zeta := 1$,
    $\mu:=\gamma$, $\varsigma:=\epsilon/2$, $\ell:=\frac{\varsigma}{\mu v_\star}$. The complex eigenvalues are displayed as the parameter $\gamma$ varies over $[0,1]$. Colors follow a rainbow scheme: red for high $\gamma$ and violet for low $\gamma$
    (in the situation in which we are reducing~$\gamma$, the image
    must be interpreted ``from red to violet'').
    The instability established in Theorem~\ref{TR5r3ee}
    arises here when the larger real eigenvalues crosses the origin from left to right.}
    \label{Fe2IG1qwefgb0215rnb78X.90}
\end{figure}

\begin{appendix}

\section{Cubic polynomials and a useful algebraic observations}\label{APP;a}

\begin{lemma}\label{LE:A}
Let~$a_0$, $a_1$, $a_2\in(0,+\infty)$ and consider the polynomial $$P(t) = t^3 + a_2 t^2 + a_1 t + a_0.$$ Then, all the roots of~$P$ have negative real part if and only if $a_1 a_2 > a_0 $.
\end{lemma}

\begin{proof}
This is in fact a particular case of the Routh-Hurwitz Stability Criterion, but we provide
here a self-contained proof for the facility of the reader.

To prove Lemma~\ref{LE:A}, we consider the complex roots~$r_1$, $r_2$, and~$r_3$ of~$P$
and the corresponding factorization
$$P(t)=(t-r_1)(t-r_2)(t-r_3),$$
leading to
\begin{equation*}
\begin{split}&
a_2=-(r_1+  r_2 + r_3) ,\\&
a_1=r_1 r_2 + r_1 r_3 + r_2 r_3 , \\ {\mbox{and }}\quad&a_0=-r_1 r_2 r_3 .
\end{split}\end{equation*}

As a result,
\begin{equation}\label{mag}
a_0-a_1a_2=(r_1 + r_2) (r_1 + r_3) (r_2 + r_3).
\end{equation}

We also remark that~$P(+\infty)=+\infty$ and~$P(-\infty)=-\infty$, therefore
\begin{equation}\label{ser.pqkdoewvnb9m87}
{\mbox{$P$ has at least one real root.}}\end{equation} For this reason, without loss of generality, we can suppose that~$r_1$ is real.

Also, since~$P(t)\ge a_0>0$ for all~$t\ge0$, we see that
\begin{equation}\label{ser}
{\mbox{all real roots must be negative}}\end{equation} and, in particular, $r_1<0$.

Now, suppose that all the roots are real and negative,
namely~$r_1$, $r_2$, $r_3\in(-\infty,0)$.
Then, it holds that~$r_i+r_j<0$ for all~$i$, $j\in\{1,2,3\}$ and thus
we infer from~\eqref{mag} that
\begin{equation}\label{Swq90nc79m0bv37}
a_0-a_1a_2<0.\end{equation}

Let us now suppose that~$P$ has complex roots.
In this case, the roots must be conjugated, since the coefficients of~$P$ are real,
hence, in this case, one can write~$r_2=x+iy$ and~$r_3=x-iy$, with~$x$, $y\in\R$,
and we obtain from~\eqref{mag} that, as long as~$x\ne0$,
\begin{equation}\label{ser2} \frac{a_0-a_1a_2}{2x}=(r_1 + x+iy) (r_1 + x-iy)=
2 r_1 x + r_1^2 + x^2 + y^2.\end{equation}

Hence, if~$P$ possesses complex roots with negative real part, we deduce from~\eqref{ser2} that
$$ \frac{a_0-a_1a_2}{2x}\ge 2 r_1 x>0$$
and therefore~$a_0-a_1a_2<0$.

These observations show that if~$P$ has complex
roots with negative real part, then necessarily~$a_1 a_2 > a_0 $.

Suppose now that~$a_1 a_2 > a_0 $. Our goal is to show that all the roots have negative
real part. The claim is true if all the roots are real, due to~\eqref{ser}, therefore
we can suppose that~$P$ has complex roots. In this situation, by~\eqref{ser2},
$$\frac{a_0-a_1a_2}{2x}=
(r_1 + x)^2 + y^2\ge0
$$and therefore~$x<0$, as desired.

The proof of Lemma~\ref{LE:A} is thereby complete.
\end{proof}

We also recall the following simple observation:

\begin{lemma}\label{LEMMA2}
Let~$a_0$, $a_1$, $a_2\in\R$ and consider the polynomial $$P(t) = t^3 + a_2 t^2 + a_1 t + a_0.$$ Then, there exists a nonzero
purely imaginary root of~$P$ if and only if $a_1>0$ and~$a_1 a_2 = a_0 $.

Also, in this case, the roots of~$P$ are~$-a_2$,  $i\sqrt{a_1}$, and~$-i\sqrt{a_1}$.
\end{lemma}

\begin{proof}
To prove this, on the one hand one observes that if~$a_1 a_2 = a_0 $ then
$$P(t) = t^3 + a_2 t^2 + a_1 t + a_1a_2=
(t^2+a_1)(t+a_2).$$
As a result, when~$a_1>0$, we get that~$P$ has purely imaginary roots~$\pm i\sqrt{a_1}$.

On the other hand, if~$P$ has a nontrivial purely imaginary root~$ir$, with~$r\in\R\setminus\{0\}$,
since the coefficients of~$P$ are real we know that~$-ir$ is a root as well,
and the other root of~$P$, say~$-c$, must be necessarily real.
As a result,
$$P(t)=(t-ir)(t+ir)(t+c)=
t^3 + c t^2 +r^2 t + c r^2 $$
and thus
$$ a_2=c,\qquad a_1=r^2>0,\qquad{\mbox{and}}\qquad a_0=cr^2,$$
yielding that~$a_1 a_2 = cr^2=a_0$.
\end{proof}

Another useful variation of the previous result is presented here below:

\begin{lemma}\label{APK0-2CMAswqpr}
Let~$a_0$, $a_1$, $a_2\in(0,+\infty)$ and consider the polynomial $$P(t) = t^3 + a_2 t^2 + a_1 t + a_0.$$ 
Suppose that~$a_1 a_2 < a_0 $.
Then, $P$ possesses one real, negative root and two complex roots with positive real part.
\end{lemma}

\begin{proof} In light of~\eqref{ser.pqkdoewvnb9m87} we know that~$P$ has at least one real root.
Moreover, by~\eqref{ser} we have that this real root must be negative.

If~$P$ had three real roots, then~\eqref{ser} would also say that all these real roots must be negative, and therefore~\eqref{Swq90nc79m0bv37} would be in force, thus giving a contradiction with our assumption.

Hence, $P$ presents one real, negative root and two complex roots.

The real part of these roots cannot be zero, due to Lemma~\ref{LEMMA2},
and cannot be negative, due to Lemma~\ref{LE:A}, hence it is necessarily positive.
\end{proof}

\section{{F}rom~\eqref{PROVV2} to~\eqref{NUOVO7}}\label{0apjclrgh}

Here we show how to obtain the differential operator in~\eqref{NUOVO7} as a proxy
for the kernel interaction in~\eqref{PROVV2}. For this purpose, we recall Pizzetti's Formula
(see e.g.~\cite[equation~(1.1.16)]{DV24BK}), according to which, for every smooth function~$f$ and~$N\in\N$,
$$\int_{\partial B_\rho(p)} f(\varphi)\,d\Sigma_\varphi=\sum_{j=0}^N C_{n,j} \,\rho^{n-1+2j}\,\Delta^j f(p)+O(\rho^{n+2N}).$$
Here above, ``$d\Sigma$'' stands for the hypersurface measure on the sphere~$\partial B_\rho(p)$ and~$C_{n,j}$
are positive constants (that can be made explicit). 

We also use the radial notation~$K(y)=K_0(|y|)$ for some function~$K_0:[0,+\infty)\to[0,+\infty)$
and use polar coordinates to conclude that
\begin{eqnarray*}
&& \int_{\R^n} K(y)\,v(x+y)\,dy=\int_0^{+\infty}\left[
\int_{\partial B_\rho} K_0(\rho)\,v(x+\rho\varphi)\,d\Sigma_\varphi\right]\,d\rho\\&&\qquad=\int_0^{+\infty}\left[ 
\sum_{j=0}^N C_{n,j} \rho^{n-1+2j}\,\Delta^j v(x)+O(\rho^{n+2N})
\right]\,K_0(\rho)\,d\rho.
\end{eqnarray*}The expression in~\eqref{NUOVO7} now
follows by assuming that~$K_0$ decays sufficiently rapidly, interchanging the integral and summation, and formally neglecting higher-order terms.

\end{appendix}

\end{document}